\documentclass[11pt]{amsart}
\usepackage{amsmath,amssymb,amsfonts,amsthm,color,latexsym}
\usepackage{graphicx} 

\usepackage{xcolor}

\newcommand{\C}{\mathbb{C}}
\newcommand{\F}{\mathcal{F}}
\newcommand{\G}{\mathcal{G}}
\newcommand{\B}{\mathcal{B}}
\newcommand{\ord}{\text{ord}}

\newtheorem{theorem}{Theorem}[section]
\newtheorem{example}[theorem]{Example}
\newtheorem{remark}[theorem]{Remark} 

\newtheorem{lemma}[theorem]{Lemma}

\newtheorem{corollary}[theorem]{Corollary}

\newtheorem{maintheorem}{Theorem}

\title{An upper bound for the GSV-index of a foliation}
\author[A. Fern\'andez-P\'erez]{Arturo Fern\'andez-P\'erez}
\address[Arturo Fern\'{a}ndez P\'erez] {Department of Mathematics. Federal University of Minas Gerais. Av. Ant\^onio Carlos, 6627 
CEP 31270-901\\
Pampulha - Belo Horizonte - Brazil. ORCID: 0000-0002-5827-8828}
\email{fernandez@ufmg.br}

\author[E. R.  Garc\'{i}a Barroso]{Evelia R. Garc\'{i}a Barroso}
\address[Evelia R. Garc\'{i}a Barroso]{Dpto. Matem\'{a}ticas, Estad\'{\i}stica e Investigaci\'on Operativa\\
IMAULL\\
Universidad de La Laguna. Apartado de Correos 456. 38200 La Laguna, Tenerife, Spain. ORCID: 0000-0001-7575-2619}
\email{ergarcia@ull.es}

\author[N. Saravia-Molina]{Nancy Saravia-Molina}
\address[Nancy Saravia-Molina]{Dpto. Ciencias - Secci\'{o}n Matem\'{a}ticas, Pontificia Universidad Cat\'{o}lica del Per\'{u}, Av. Universitaria 1801,
San Miguel, Lima 32, Peru. ORCID: 0000-0002-2819-8835}
\email{nsaraviam@pucp.edu.pe}



\begin{document}

\begin{abstract}
Let $\F$ be a holomorphic foliation at $p\in\C^2$, and let $B$ be a separatrix of $\F$. We prove the following upper bound $GSV_p(\F,B)\leq 4\tau_p(\F,B)-3\mu_p(\F,B)$, where $GSV_p(\F,B)$ is the \textit{G\'omez-Mont-Seade-Verjovsky index} of the foliation $\F$ with respect to $B$, $\mu_p(\F,B)$ is the multiplicity of $\F$ along $B$ and $\tau_p(\F,B)$ is the dimension of the quotient of $\mathbb C\{x,y\}$ by the ideal generated by the components of any $1$-form defining $\F$ and any equation of $B$. 
\end{abstract}
\maketitle

\section{Introduction}

Let $\F$ be a germ of singular holomorphic foliation at $(\C^{2},p)$ given by the $1$-form $\omega:= P(x,y)dx+Q(x,y)dy$, where $P(x,y),Q(x,y)\in \mathbb C\{x,y\}$, and $B$ be a separatrix of $\F$. Several numerical invariants can be studied for the pair $(\F,B)$, such as the  {\it  G\'omez-Mont--Seade--Verjovsky index} (see \cite{GSV}) of $\F$ with respect to  $B$, denoted by $GSV_p(\F,B)$, the {\it multiplicity of $\F$ along} $B$, $\mu_{p}(\mathcal{F},B)$ (see Section \ref{sec:multiplicity} for definition), and the \textit{Tjurina number} of $\F$ with respect to $B: f(x,y)=0$, defined by \[\tau_p(\F, B):=\dim_{\C}\C\{x,y\}/(P,Q,f).\]  

\par In this paper,
inspired by the blow-up techniques developed by Wang \cite{Wang}, we establish an optimal upper bound for  $GSV_P(\F,B)$ in terms of $\mu_p(\F,B)$ and $\tau_p(\F,B)$ as follows:

\begin{maintheorem}
    \label{th:main}
     Let $\F$ be a germ of a singular holomorphic foliation on $(\mathbb{C}^2,p)$, and $B$ be a separatrix of $\F$ at $p$. Then 
     \begin{equation}\label{eq_main}
    GSV_p(\F,B)\leq 4\tau_p(\F,B)-3\mu_p(\F,B)
     \end{equation}
     where $GSV_p(\F,B)$ denotes the GSV-index of $\F$ with respect to $B$.
     Moreover, the equality holds if and only if $B$ is smooth.
\end{maintheorem}
 
\par Note that if $B=\{f=0\}$ is an irreducible plane curve germ,
 then by applying Theorem \ref{th:main} to the foliation $\F: df=0$, we obtain the bound stated in \cite[Question 4.2]{Dimca-Greuel} and proved simultaneously in \cite{Alberichetals}, \cite{Genzmer-Hernandes} and \cite{Wang}, since $\mu_p(\F,B)=\mu(B)$, $\tau_p(\F,B)=\tau(B)$, and $GSV_p(\F,B)=0$ in this case. Here, $\mu(B)$ and $\tau(B)$ denote the classical Milnor and Tjurina numbers of $B$. A complete answer to this question was given by Almir\'on \cite{Almiron}. Moreover, the author proposed a broader perspective of the study of the difference between Milnor and Tjurina numbers for isolated complete intersection singularities. 

\section{Multiplicity of a foliation along a divisor of separatrices}
\label{sec:multiplicity}

Throughout this note $\F$ denotes
a germ of a singular holomorphic foliation at $(\mathbb C^2,p)$, given in local coordinates
$(x,y)$ centered at $p$ by a  1-form $\omega:= P(x,y)dx+Q(x,y)dy$, where
$P(x,y),Q(x,y)\in \mathbb C\{x,y\}$ are coprime; or by  its dual vector field
\[
v:=-Q(x,y)\frac{\partial}{\partial x}+P(x,y)\frac{\partial}{\partial y}.
\]

The {\it algebraic multiplicity} $\nu_p(\F)$ of $\F$ is
the minimum of the orders $\ord_p(P )$ and $\ord_p(Q)$.
Remember that a plane curve germ $f(x,y)=0$ is an {\it $\F-$invariant curve} if $\omega \wedge df=(f.h)dx \wedge dy$, for some $h(x,y)\in \mathbb C\{x,y\}$ and a {\it separatrix} of $\F$ is an irreducible $\F$-invariant curve.

\par Let $\F$ be a germ of a singular foliation at $(\C^2,p)$ induced by the vector field $v$ and $B$ be a separatrix of $\F$. Let $\gamma:(\C,0)\to(\C^2,p)$ be a primitive parametrization of $B$, the \textit{multiplicity of $\F$ along $B$ at $p$} is by definition 
\begin{equation}\label{eq_mult}
\mu_p(\F,B)=\ord_t \theta(t),
\end{equation}
where $\theta(t)$ is the unique vector field at $(\C,0)$ such that 
$\gamma_{*} \theta(t)=v\circ\gamma(t)$, see for instance \cite[p. 159]{CLS}. If $\omega=P(x,y)dx+Q(x,y)dy$ is a 1-form inducing $\F$ and $\gamma(t)=(x(t),y(t))$, then 
\begin{equation}\label{mult_2}
\theta(t)=
\begin{cases}
-\frac{Q(\gamma(t))}{x'(t)} & \text{if $x(t)\neq 0$}
\medskip \\
 \frac{P(\gamma(t))}{y'(t)} & \text{if $y(t)\neq 0$}.
\end{cases}
\end{equation}

We extend the notion of multiplicity of  $\mathcal{F}$ along any nonempty divisor $\mathcal{B}=\sum_Ba_B B$ of separatrices of $\mathcal F$ at $p$ in the following way:
\begin{equation}\label{eq:gmul}
\mu_{p}(\mathcal{F},\mathcal{B})=\left(\displaystyle\sum_{B}a_{B}\mu_{p}(\mathcal{F},B)\right)-\deg(\mathcal{B})+1.
\end{equation}
By convention we put $\mu_{p}(\mathcal{F},\mathcal{B})=1$ for any empty divisor $\mathcal{B}$.
In particular, when $\mathcal{B}=B_1+\cdots+B_r$ is an effective divisor of separatrices of the singular foliation $\mathcal{F}$ at $p$, we rediscover \cite[Equation (2.2), p. 329]{K-S} (for the reduced plane curve $C:=\cup_{i=1}^rB_i$).

\noindent Hence, if we write $\mathcal B=\mathcal B_0-\mathcal B_{\infty}$ where $\mathcal B_0$ and $\mathcal B_{\infty}$ are effective divisors we get

\[
\mu_{p}(\mathcal{F},\mathcal{B})=\mu_{p}(\mathcal{F},\mathcal{B}_0)-\mu_{p}(\mathcal{F},\mathcal{B}_{\infty})+1.
\]

\noindent Since $\mu_{p}(\mathcal F, \mathcal B_{\infty})\geq 1$ we have

\[
\mu_{p}(\mathcal{F},\mathcal{B}_0)\geq \mu_{p}(\mathcal{F},\mathcal{B}).
\]

Denote by $GSV_p(\F,\B)$ the {\it G\'omez-Mont-Seade-Verjovsky index} of the foliation $\F$ at $(\mathbb C^2,p)$ (GSV-index) with respect to the effective divisor $\B$ of separatrices of  $\F$. See \cite{GSV} for details.

\begin{remark}\label{re:irr}
Let $B:f(x,y)=0$ be a separatrix of $\F$ and let $\mathcal{G}_f$ be the hamiltonian foliation defined by $df=0$. From \eqref{eq_mult} and \eqref{mult_2}, we have

\[
\mu_p(\G_f,B)=
\begin{cases}
\ord_t \partial_y{f}(\gamma(t))-\ord_t x(t)+1 & \text{if $x(t)\neq 0$}
\medskip \\
 \ord_t \partial_x{f}(\gamma(t))-\ord_t y(t)+1 & \text{if $y(t)\neq 0$},
\end{cases}
\]

\noindent where $\gamma(t)=(x(t),y(t))$ is a Puiseux parametrization of $B$. By \cite[Proposition 5.7]{FP-GB-SM2021}, we get 
$\mu_p(\F,B)=GSV_p(\F,B)+\mu_p(\mathcal{G}_f,B)$. But, since $B$ is irreducible, after  \cite[Proposition 4.7]{FP-GB-SM2021}, $\mu_p(\G_f,B)=\mu_p(B)$  and so that 
$\mu_p(\F,B)=GSV_p(\F,B)+\mu_p(B).$ 
Hence, according to \cite[Proposition 6.2]{FP-GB-SM2021}, we obtain 
\begin{equation}\label{ec}
\mu_p(\F,B)-\tau_p(\F,B)=\mu_p(B)-\tau_p(B).
\end{equation}
Consequently when $B$ is quasihomogeneous by (\cite{Saito}, Satz p.123 a) and d)] and \eqref{ec} or by \cite[Theorem 4]{Zariski}, $B$ is analytically equivalent to $y^{n}+x^{m}=0$ for some natural coprime numbers $n$, $m >1$,  in this case,  we get  $\mu_p(\F,B)=\tau_p(\F,B)$. 
\end{remark}

\section{Proof of Theorem \ref{th:main}}
\label{sec:proofmain}
Let $C$  be an irreducible plane curve of multiplicity $\nu_p(C)=n$. Denote by $\tilde C$ the strict transform of $C$ by a blow-up  and $\bar C$ the normalization of $C$. We have natural morphisms
 \begin{equation}\label{eqW1}
\pi: \bar C \buildrel \rho\over \longrightarrow \tilde C \buildrel \sigma \over\longrightarrow C.
 \end{equation}
 Consider $(C,p)
 \buildrel i\over\longrightarrow
  (\mathcal C,x_0)  \buildrel \phi\over\longrightarrow (T,t_0)$ the miniversal deformation of $C$. We extend to $\mathcal C$ the morphisms in \eqref{eqW1}:

  \[
\pi: \bar{\mathcal C} \buildrel \rho\over \longrightarrow \tilde{\mathcal C} \buildrel \sigma \over\longrightarrow \mathcal C.
\]

We have two other  natural morphisms $\tilde C \buildrel \tilde \phi \over\longrightarrow T$ and $\bar C \buildrel \bar \phi \over\longrightarrow T$ and its respectively extensions to the deformation, that is $\tilde{\mathcal C} \buildrel \tilde \phi \over\longrightarrow T$ and $\bar{\mathcal C} \buildrel \bar \phi \over\longrightarrow T$.\\
  
  Let $\Omega_{\mathcal C/T}$ (respectively, $\Omega_{\tilde{\mathcal C}/T}$, respectively, $\Omega_{\bar{\mathcal C}/T}$ ) denote the $\mathcal O_{\mathcal C}$-module of relative K\"ahler differentials of $\mathcal C$ over $T$ (respectively, the $\mathcal O_{\tilde{\mathcal C}}$-module of relative K\"ahler differentials of $\tilde{\mathcal C}$ over $T$, respectively the $\mathcal O_{\bar{\mathcal C}}$-module of relative K\"ahler differentials of $\bar{\mathcal C}$ over $T$).  Put $G:=\rho^*\Omega_{\tilde{\mathcal C}/T}/\pi^*\Omega_{{\mathcal C}/T}$. If $C$ is singular then  $G$ is not zero and  $\bar\phi_* G$ is a finite $\mathcal O_T$-module and after Wang (\cite{Wang}) we put $\mathcal D(t)=\hbox{\rm length}_{\mathcal O_{T,t}}((\bar\phi_* G)_t)$ and $\alpha:=\min_{t\in T} \mathcal D(t)$. 
    By \cite[Claim 4.4]{Wang} we have
\begin{equation}\label{in: alpha}
    \alpha >\dfrac{\nu_{p}(C)(\nu_{p}(C)-1)}{4},\,\,\,\,\text{for}\,\,\,\,  \nu_{p}(C) \geq 2.
\end{equation}

\par We will need the following lemma.
\begin{lemma} \label{eq:rMTG} 
Let $\F$ be a germ of a singular foliation on $(\mathbb{C}^2,p)$, and $B$ be a separatrix of $\F$ at $p$. If $\tilde \F$ (respectively  $\tilde B$) is the strict transform of $\F$ (respectively of $B$) by $\sigma$ at the point $q$, then 
\begin{equation}\label{eq:SW}
3\mu_q(\tilde \F,\tilde B)-4\tau_q(\tilde \F,\tilde B)+GSV_q(\tilde \F,\tilde B)\geq 3\mu_p(\F,B)-4\tau_p(\F,B)+GSV_p(\F,B).
\end{equation}
Moreover, the equality holds if and only if $B$ is smooth.
\end{lemma}
\begin{proof}
Suppose first $\nu_p(B)\geq 2$. By \cite[Equality (1.2), p. 291]{Car} we have
\begin{equation}\label{eq1N}
\mu _{q}(\tilde{\mathcal{F}},\tilde B)=
\mu _{p}(\mathcal{F},B)-\nu_{p}(B)(m_{p}(\mathcal{F})-1),
\end{equation}
\noindent where 

\[
m_p(\F)=\left\{\begin{array}{ll}
\nu_p(\F)+1 & \hbox{\rm if $\sigma$ is dicritical}\\
\nu_p(\F) & \hbox{\rm if $\sigma$ is nondicritical}.\\
\end{array}
\right.
\]

\noindent On the other hand by the behavior under blow-up of the GSV index \cite[p. 30]{Brunella-book} we get
\begin{equation}\label{GSVb}
GSV_q(\tilde{\F},\tilde B)=GSV_p(\F,B)+\nu_p(B)(\nu_p(B)-m_p(\F)),
\end{equation}
and after \cite[Proposition 6.2]{FP-GB-SM2021}  we have
\begin{equation}\label{GSVc}
 GSV(\F,B)=\tau_p(\F,B)-\tau_p(B).
\end{equation}
\noindent From \cite[Equation (4)]{Wang} we obtain
\begin{equation}\label{eqW}
\tau_p(B)-\tau_q(\tilde B)\geq\frac{\nu_p(B)(\nu_p(B)-1)}{2}+\alpha, 
\end{equation}
\noindent where  $\alpha$ was defined in \eqref{in: alpha}. From \eqref{GSVc}, \eqref{GSVb}, and \eqref{eqW} we deduce
\begin{equation}\label{eq3}
\tau_{q}(\tilde{\mathcal{F}},\tilde{B})\leq\tau_{p}(\mathcal{F},B)+\dfrac{\nu_p(B)[\nu_{p}(B)+1-2m_p(\F)]}{2}- \alpha.
\end{equation}
\noindent After \eqref{eq1N}, \eqref{GSVb} and \eqref{eq3} we get
\begin{eqnarray}\label{eq:estricta}
3\mu_q(\tilde \F,\tilde B)-4\tau_q(\tilde \F,\tilde B)+GSV_q(\tilde \F,\tilde B) &\geq& 3\mu_p(\F,B)-4\tau_p(\F,B)+GSV_p( \F,B)  \nonumber \\ 
&  + & 4\alpha- \nu_p(B)(\nu_p(B)-1)\nonumber \\
&>&3\mu_p(\F,B)-4\tau_p(\F,B)+GSV_p( \F,B),
\end{eqnarray}
where the last inequality follows by \eqref{in: alpha}, since $\nu_{p}(B)\geq 2.$

Suppose now that  $B$ is a non-singular curve. By a change of coordinates, if necessary,  we can suppose that $B$ is given by $x=0$. Since $B$ is a separatrix of $\mathcal{F}:\omega=Pdx+Qdy$ then $\omega \wedge dx = Q dx \wedge dy$. Hence $Q=xh$ for some convergent power series $h(x,y)\in \mathbb C\{x,y\}$  such that $h(0,y)\not = 0$ and $\mathcal{F}:\omega=P(x,y)dx+xh(x,y)dy$. 
Since $\gamma(t)=(0,t)$ is a parametrization of $B$, then $\mu_p(\mathcal{F},B)=\ord_{t}P(0,t)$ and
\begin{equation}\label{eq:smooth}
 \tau_p(\mathcal{F},B)=\dim_{\C}\C\{x,y\}/(P,xh,x)=\dim_{\C}\C\{x,y\}/(P,x)=\mu_p(\mathcal{F},B).   
\end{equation}
 Moreover
by \eqref{GSVc} we have the equality $GSV_p(\F,B)=\tau_p(\F,B)$.
Thus, if $\nu_p(B)=1$  then $\nu_p(\tilde B)=1$, $\mu_p(\F, B)=\tau_p(\F,B)=GSV_p( \F,B)$ and 
$\mu_p(\tilde \F, \tilde B)=\tau_p(\tilde\F,\tilde B)=GSV_p(\tilde \F,\tilde B)$, so \eqref{eq:SW} becomes an equality.
Moreover, if \eqref{eq:SW} is an equality, then $B$ is necessarily smooth. Indeed, if $B$ were singular, by \eqref{in: alpha} the number $-\nu_p(B)(\nu_p(B)-1)+4\alpha$ is positive and from the equality \eqref{eq:estricta} we have that \eqref{eq:SW} is a strict inequality that is a contradiction. Hence the lemma follows.
\end{proof}

\begin{remark}\label{relations}
By definition, for any branch $B$ we have  $\tau_p(B)\leq \mu_p(B)$, so after \cite[Proposition 6.2]{FP-GB-SM2021} and \cite[Propositon 5.7]{FP-GB-SM2021} we get $\tau_p(\F,B)-GSV_p(\F,B)\leq \mu_p(\F,B)-GSV_p(\F,B)$ and 
\begin{equation}\label{eq:desTM}
\tau_p(\F,B)\leq \mu(\F,B).
\end{equation}
In addition, by \eqref{eq:smooth}, if $B$ is smooth then $\tau_p(\F,B)=\mu(\F,B)$.
\end{remark}

We will now prove the main theorem of this paper.
\subsection{Proof of  Theorem \ref{th:main}}
    
    Suppose that $\F$ is a holomorphic foliation at $p\in \mathbb C^{2}$, $B$ is a separatrix of $\F$ at $p$. First, assume that $B$ is singular, that is, $\nu_p(B)\geq 2$. Then, there is a sequence of  blow-ups of  $\mathcal{F}$ at $p$ (cf. Seidenberg \cite{seidenberg}):
 \begin{equation}
(\mathcal{F}^{(N)}, B^{(N)})   \longrightarrow \cdots  \longrightarrow (\mathcal{F}^{(2)}, B^{(2)}) \longrightarrow (\mathcal{F}^{(1)}, B^{(1)})  \longrightarrow (\mathcal{F}^{(0)},B^{(0)})=(\mathcal{F},B), 
 \end{equation}  
 where $B^{(i+1)}$ denotes the strict transform of $B^{(i)}$ thought $q^{(i)}$ for each $i=0,\ldots, N-1$, $B^{(N)}$ is smooth, and $\F^{(N)}$ is a foliation with a simple singularity at $q^{(N)}\in B^{(N)}$, or $\F^{(N)}$ is a regular foliation at $q^{(N)}$, and $B^{(N)}$ is a dicritical separatrix (cf. \cite[p. 1423]{Genzmer-Mol}). 
 Put $\mu^{(i)}:=\mu_{{q^{(i)}}}(\mathcal{F}^{(i)}, B^{(i)})$ and $\tau^{(i)}:=\tau_{{q^{(i)}}}(\mathcal{F}^{(i)}, B^{(i)})$ for any $i\in\{1,\ldots,N\}$.
 
\par Applying Lemma \ref{eq:rMTG} to each blow-up  $(\mathcal{F}^{(i+1)}, B^{(i+1)}) \longrightarrow (\mathcal{F}^{(i)}, B^{(i)})$, 
we get 
\[
3\mu^{(i)}-4\tau^{(i)}+GSV_{q^{(i)}}(\mathcal{F}^{(i)},B^{(i)})\leq 3\mu^{(i+1)}-4\tau^{(i+1)}+GSV_{q^{(i+1)}}(\mathcal{F}^{(i+1)},B^{(i+1)}).
\]
 In particular
\[
3 \mu_{p}(\mathcal{F},B) -4\tau_{p}(\mathcal{F},B)+GSV_p(\mathcal{F},B)<3\mu^{(N)}-4\tau^{(N)}+GSV_{q^{(N)}}(\mathcal{F}^{(N)},B^{(N)})=0,
\]
where the last equality holds, since $B^{(N)}$ is a smooth curve. 
\par If $B$ is a smooth curve then, by Remarks \ref{relations}, and \ref{re:irr}, $GSV_p(\F,B)=\mu_{p}(\mathcal{F},B)=\tau_{p}(\mathcal{F},B)$, then the inequality (\ref{eq_main}) is an equality. Reciprocally,   if (\ref{eq_main}) is an equality, then $B$ is necessarily smooth. Indeed, if $B$ were singular, by Lemma \ref{eq:rMTG}, $3\mu_p(\F,B)-4\tau_p(\F,B)+GSV_p(\F,B)<0$, this is a contradiction. Hence, the theorem follows. 
$\Box$
\section{Examples}
In this section, we provide two examples of holomorphic foliations and separatrices that illustrate Theorem \ref{th:main}.

\begin{example} \label{ex2}
Let $\G_{m,n}$ be the germ of holomorphic foliation at $(\C^2,0)$ defined by 
\[\eta_{m,n}=mxdy-nydx,\] where $1\leq m\leq n$ are integers.  The curve $B: y^m-x^n=0$ is a separatrix of $\G_{m,n}$ at $0\in\C^2$ and 
\begin{eqnarray*}
GSV_0(\G_{m,n},B)&=&m+n-mn, \\
 \mu_0(\G_{m,n},B)&=& 1,\\
 \tau_0(\G_{m,n},B)&=& 1.
\end{eqnarray*}
We get \[3\mu_0(\G_{m,n},B)-4\tau_0(\G_{m,n},B)+GSV_0(\G_{m,n},B)=m+n-mn-1=-(m-1)(n-1)\leq 0.\] Thus, Theorem \ref{th:main} is verified. 
\end{example}

\begin{example}\label{ex:teo1}
Let $\F_m$ be the germ of holomorphic foliation at $(\C^2,0)$ defined by 
\[\omega_m=\left((2m+1)yx^{m+1}+my^{m+2}\right)dx+\left((1-m)xy^{m+1}-2mx^{m+2}\right )dy,\] where $m\in\mathbb{Z}_{>0}$.  The curve $B: x^{2m+1}+x^{m}y^{m+1}+y^{2m}=0$
 is a separatrix of $\F$ at $0\in\C^2$. We get
\[
GSV_0(\F_m,B)=\left\{\begin{array}{cl}
3&\text{if $m=1$}\\
-2m^2+4m+1& \text{ if $m>1$},
\end{array}
\right.
\]
\[
\mu_0(B)=2m(2m-1)\,\,\,\,\,\,\,\,\,\,m\geq 1,
\]
\[
\tau_0(B)=\left\{\begin{array}{cl}
2&\text{if $m=1$}\\
3m^2& \text{ if $m>1$},
\end{array}
\right.
\]
\[
 \mu_0(\F_m,B)=\left\{\begin{array}{cl}
 5&\text{ if $m=1$}\\
 2m^{2}+2m+1&\text{ if $m>1$}
 \end{array}
\right.
 \]
 
 and
 
 \[
 \tau_0(\F_m,B)=\left\{\begin{array}{cl}
 5&\text{ if $m=1$}\\
 m^{2}+4m+1&\text{if $m>1$}.
 \end{array}
\right.
\]

A straightforward calculation reveals that 
\[3\mu_0(\F_m,B)-4\tau_0(\F_m,B)+GSV_0(\F_m,B)=\left\{\begin{array}{cl}
-2&\text{if $m=1$}\\
-6m& \text{ if $m>1$}.
\end{array}
\right.
\]
Therefore, Theorem \ref{th:main} is verified. 
\end{example}
\section{Applications}

We define the sum  
 $\displaystyle T_p(\F,C):=\sum_{B\subset C}\tau_p(\F,B)$ of Tjurina numbers of $\F$ along the irreducible components of $C$, then we get the following inequality for foliations:

\begin{corollary}
\label{coro:general}
   Let $\F$ be a germ of a singular  holomorphic foliation  at $(\C^2,p)$. Let $\B$ be an effective primitive balanced divisor of separatrices for $\F$ at $p$. Then 
    \begin{equation}\label{T}
    GSV_p(\F,\B)<4 T_p(\F,\B)-3\mu_p(\F,\B)
    \end{equation}
\end{corollary}
\begin{proof} 
Suppose that $\B=\sum_{j=1}^{\ell} B_j$, where each $B_j$ is a separatrix for $\F$. 
It follows from Theorem \ref{th:main} that \[\sum_{j=1}^{\ell}\left[3\mu_p(\F,B_j)-4\tau_p(\F,B_j)+GSV_p(\F,B_j)\right]\leq -\ell.\]

\noindent Thus 
\[3\sum_{j=1}^{\ell}\mu_p(\F,B_j)-4T_p(\F,\B)+\sum_{j=1}^{\ell}GSV_p(\F,B_j)\leq-\ell\]

\noindent and, using  \eqref{eq:gmul} together with the GSV-index addition formula (\cite[p. 29]{Brunella-book}), we obtain

\[3\left(\mu_p(\F,\B)+\ell-1\right)-4T_p(\F,\B)+GSV_p(\F,\B)+2\sum_{i\neq k}i_p(B_i,B_k)\leq-\ell,\]

\noindent so 

\[3\mu_p(\F,\B)-4T_p(\F,\B)+GSV_p(\F,\B)\leq-4\ell+3-2\sum_{i\neq k}i_p(B_i,B_k).\]
Since $\ell\geq 1$, we get
\[3\mu_p(\F,\B)-4T_p(\F,\B)+GSV_p(\F,\B)<0.\]
\end{proof}

The following example illustrates Corollary \ref{coro:general}.

\begin{example}\label{ex:cor}
Let $\F:\omega=(2x^7+5y^5)dx-xy^2(5y^2+3x^5)dy$, whose leaves are contained in the connected components of the curves $C_{\lambda,\zeta}:\zeta(y^5-x^7+x^5y^3)-\lambda x^5=0$, where $(\zeta,\lambda)\neq(0,0)$. Let $C: x=0$, and consider the effective primitive balanced divisor $\B= C+C_{1,1}$ of separatrices for $\F$ at $0\in \mathbb C^{2}$. We have $GSV_0(\F,C)=GSV_0(\F,C_{1,1})=5$,  so \[GSV_0(\F,\B)=GSV_0(\F,C)+GSV_0(\F,C_{1,1})-2i_0(C,C_{1,1})=5+5-2\cdot 5=0,\] 
$\mu_0(\F,\B)=\mu_0(\F,C)+\mu_0(\F,C_{1,1})-deg(\B)+1=5+21-2+1=25$, and $T_0(\F,\B)=\tau_0(\F,C)+\tau_0(\F,C_{1,1})=5+21=26$. Hence
\[GSV_p(\F,\B)=0<4T_p(\F,\B)-3\mu_p(\F,\B)=29.\]
\end{example}

\noindent We conclude the paper with a question:

\noindent {\bf {Question:}} Is the inequality \eqref{T} true if instead of $T_{p}(\mathcal{F},\B)$ we write $\tau_{p}(\mathcal{F},\B)$?

\bigskip

\noindent {\it Acknowledgments.}
The authors gratefully acknowledge the support of Universidad de La Laguna (Tenerife, Spain), where part of this work was done.

\medskip

\noindent {\it Funding.}
This work was funded by the Direcci\'on de Fomento de la Investigaci\'on at the PUCP through grant DFI-2023-PI0979 and  by the Spanish grant PID2019-105896GB-I00 funded by MCIN/AEI/10.13039/501100011033.

\medskip

\noindent {\it Conflict of interest.}
The authors of this work declare that they have no conflicts of interest.


\begin{thebibliography}{99}

\bibitem{Alberichetals} M. Alberich-Carrami\~nana, P. Almir\'on, G. Blanco and A. Melle-Hern\'andez. 
The minimal Tjurina number of irreducible germs of plane curve singularities,
Indiana Univ. Math. J. 70 (2021), no. 4, 1211-1220. 

\bibitem{Almiron} P. Almir\'on. On the quotient of Milnor and Tjurina numbers for two-dimensional isolated hypersurface singularities, Mathematische Nachrichten 295 (2022), 1254-1263.

\bibitem{Brunella-book} M. Brunella. Birational geometry of foliations. IMPA Monogr., 1. Springer, Cham, 2015. xiv+130 pp.

\bibitem{CLS} C. Camacho, A. Lins Neto, and P. Sad.
Topological invariants and equidesingularization for
holomorphic vector fields. J. Differential Geom. 20 (1) (1984), 143-174.

\bibitem{Car} Carnicer, Manuel M.
The Poincar\'e problem in the nondicritical case.
Ann. of Math. (2) 140 (1994), no.2, 289–294.

\bibitem{Dimca-Greuel}
A. Dimca, G.-M. Greuel: On 1-forms on isolated complete intersection on curve singularities. J. of Singul. 18 (2018), 114-118. 

\bibitem{FP-GB-SM2021} A. Fern\'andez-P\'erez, E.R. Garc\'{\i}a-Barroso, and N. Saravia-Molina: On Milnor and Tjurina numbers of foliations. (2021), arXiv: 2112.14519.


\bibitem{Genzmer-Hernandes} Y. Genzmer and M. E. Hernandes: On the Saito basis and the Tjurina number for plane branches. Trans. Amer. Math. Soc. 373 (2020), 3693-3707.

\bibitem{Genzmer-Mol}
Y. Genzmer and R. Mol. Local polar invariants and the Poincar\'e problem in the dicritical case. J. Math. Soc. Japan 70 (4) (2018), 1419-1451. 

\bibitem{GSV}
X. G\'omez-Mont, J. Seade and A. Verjovsky: The index of a holomorphic flow with an isolated singularity. Math. Ann. 291 (4) (1991), 735-751.

\bibitem{K-S} B. Khanedani and T. Suwa. First variation of holomorphic forms and some applications. Hokkaido Math. J.26(1997), no.2, 323-335.
 
\bibitem{Saito} K. Saito. Quasihomogene isolierte Singularit\"aten von Hyperfl\"achen. Invent. Math. 14 (1971), 123-142.

\bibitem{seidenberg}
A. Seidenberg: Reduction of singularities of the differential equation $Ady=Bdx$. Amer. J. Math. 90 (1968), 248-269.

\bibitem{Wang}
Z. Wang: Monotic invariants under blowups. International Journal of Mathematics, Vol. 31, No. 12, (2020) 2050093. https://doi.org/10.1142/S0129167X20500937

\bibitem{Zariski}
O. Zariski. Characterization of plane algebroid curves whose module of differentials has maximum torsion. Proc. Nat. Acad. Sci. 56 (1966), 781-786.

\end{thebibliography}
\end{document}